\pdfoutput=1
\documentclass[letterpaper,journal,twoside,web]{ieeecolor}
\usepackage{lcsys}
\usepackage{graphicx}           
\usepackage{pgfplots}
\pgfplotsset{compat=newest}
\usepgfplotslibrary{groupplots}
\usepackage{tikz}               
\pgfdeclarelayer{background}
\pgfsetlayers{background,main}
\usetikzlibrary{positioning}    
\usetikzlibrary{3d}
\usetikzlibrary{calc}           
\usepackage{amsmath}            
\usepackage{amssymb}            

\usepackage{amsthm,thmtools}    
\allowdisplaybreaks             
\usepackage[hidelinks]{hyperref}
\usepackage{cleveref}           
\let\labelindent\relax
\usepackage{enumitem}           
\setlist[itemize,1]{leftmargin=\dimexpr 15pt} 
\usepackage{cite}               
\usepackage{microtype}          
\usepackage{url}                
\usepackage{tcolorbox}          
\usepackage{xcolor}             
\usepackage{dsfont,BOONDOX-ds}  

\renewcommand{\theenumi}{(\roman{enumi})} 

\declaretheorem[style=definition]{theorem}
\declaretheorem[style=definition]{corollary}
\declaretheorem[style=definition]{proposition}
\declaretheorem[style=definition]{lemma}
\declaretheorem[style=definition,qed=$\vartriangle$]{remark}

\declaretheorem[style=definition]{definition}
\declaretheorem[style=definition]{fact}

\newcommand{\nn}{\nonumber}               
\newcommand{\beq}{\begin{equation}}       
\newcommand{\eeq}{\end{equation}}         
\newcommand{\bseq}{\begin{subequations}}  
\newcommand{\eseq}{\end{subequations}}    
\newcommand{\bma}{\left[}                 
\newcommand{\ema}{\right]}                

\newcommand{\N}{\mathbb{N}}              
\newcommand{\R}{\mathbb{R}}              
\newcommand{\T}{\mathbb{T}}              
\newcommand{\W}{\mathbb{W}}              
\newcommand{\B}{\mathcal{B}}             
\newcommand{\M}{\mathcal{M}}             
\renewcommand{\L}{\mathsf{L}}            
\newcommand{\LTI}{\mathsf{LTI}}          
\newcommand{\Lfh}{\L_{\textup{fin}}}     
\newcommand{\LTIfh}{\LTI_{\textup{fin}}} 

\newcommand{\Ker}{\operatorname{ker}}         
\newcommand{\Image}{\operatorname{im}}        
\newcommand{\transpose}{\mathsf{T}}           
\newcommand{\norm}[1]{\left\lVert#1\right\rVert} 
\newcommand{\Span}{\operatorname{span}}       
\newcommand{\Grass}[2]{\operatorname{Gr}(#1,#2)} 
\newcommand{\Ort}[1]{\operatorname{O}(#1)}    

\definecolor{cb-green}{RGB}{27,158,119}   
\definecolor{cb-orange}{RGB}{217,95,2}    
\definecolor{cb-purple}{RGB}{117,112,179} 

\usepackage[acronym]{glossaries}
\newacronym{LTI}{LTI}{Linear Time-Invariant}
\newacronym{AR}{AR}{Auto-Regressive}
\newacronym{MPUM}{MPUM}{Most Powerful Unfalsified Model}
\newacronym{SVD}{SVD}{Singular Value Decomposition}

\begin{document}
\title{Distances between finite-horizon linear behaviors}%
\author{A. Padoan, \IEEEmembership{Member, IEEE} and J. Coulson, \IEEEmembership{Member, IEEE}
\thanks{%
    A. Padoan is with the Department of Electrical and Computer Engineering, University of British Columbia, Vancouver, BC V6T 1Z4, Canada (email: \href{mailto:apadoan@ece.ubc.ca}{apadoan@ece.ubc.ca}). J. Coulson is with the Department of Electrical and Computer Engineering, University of Wisconsin–Madison, Madison, WI, USA ( email: \href{mailto:jeremy.coulson@wisc.edu}{jeremy.coulson@wisc.edu}).%
  }%
}%

\maketitle
\thispagestyle{empty}
\begin{abstract}
 The paper introduces a class of distances for linear  behaviors over finite time horizons.  These distances 
 allow for comparisons between finite-horizon linear behaviors represented by matrices of possibly different dimensions. They remain invariant under coordinate changes, rotations, and permutations, ensuring independence from input-output partitions. Moreover,  they naturally encode complexity-misfit trade-offs for \gls{LTI} behaviors, providing a principled solution to a longstanding puzzle in behavioral systems theory. The resulting framework characterizes  modeling as a minimum distance problem, identifying
the \gls{MPUM} as optimal among all systems unfalsified by a given dataset. Finally, we illustrate the value of these metrics in a time-series anomaly detection task, where their finer resolution yields superior performance over existing distances.
\end{abstract}

\glsreset{MPUM}
\glsreset{LTI}

\section{Introduction}
\label{sec:introduction} 

\IEEEPARstart{L}{inear} systems theory has recently witnessed renewed interest in subspace representations~\cite{coulson2019data,markovsky2021behavioral,padoan2022behavioral,fazzi2023distance,sasfi2024subspace}.  Grounded in behavioral systems theory~\cite{willems1986timeI,willems1986timeII,willems1987timeIII}, which models systems as sets of trajectories, subspace representations provide a flexible alternative to classical parametric models, characterizing all system trajectories over a finite time horizon.  While subspace representations rely on linearity, they avoid committing to a specific model structure, such as  state-space representations or transfer functions. Instead, they can operate directly on raw data and characterize linear behaviors over finite horizons, without explicit parametrizations or \textit{a priori} complexity constraints. These properties have made these representations an established foundation for subspace identification algorithms~\cite{van1996subspace} and recent advances in data-driven control~\cite{coulson2019data}, often leading to superior performance compared to traditional parametric approaches~\cite{markovsky2021behavioral}.

A natural question arising in this context is:
``How should one compare subspace representations of finite-horizon linear behaviors?''
In system identification and data-driven control, subspace representations are often compared using distances induced by matrix norms. While simple to compute, such distances give rise  to several fundamental inconsistencies:  
(i) matrix representations of different dimensions cannot be meaningfully compared—even if they describe the same behavior;  (ii) distances may depend on arbitrary coordinate choices; and (iii) behaviors that are geometrically distant may appear deceptively close. All these issues arise because finite-horizon linear behaviors are, by nature, linear subspaces~\cite{willems1986timeI,willems1986timeII,willems1987timeIII}, and should be compared using metrics that respect their geometry. Prior work has explored this issue—most notably through the $L$-gap metric and Grassmann metric studied in~\cite{padoan2022behavioral} and~\cite{fazzi2023distance}. However, both metrics are useful only when comparing subspaces of equal dimension, as they either assign a trivial maximal value or become undefined when dimensions differ.
We introduce a new class of distances inherited from the doubly infinite Grassmannian~\cite{ye2016schubert}, which resolve the inconsistencies noted above, provide a principled answer to a long-standing puzzle in behavioral systems theory, and offer improved resolution in time series analysis applications.

\textbf{Related work} Distances have long played a foundational role in systems and control theory. The gap metric~\cite{zames1981uncertainty} and its refinements~\cite{vidyasagar1984graph,vinnicombe2001uncertainty} have shaped robust control by providing a topology for comparing feedback systems. In the context of behavioral systems theory for \gls{LTI} systems, notions of distance have been explored in~\cite{sasane2003distance,trentelman2013distance}, while subspace angles have been studied in~\cite{de2002subspace,roorda2001optimal}. A misfit–complexity trade-off for identification was proposed in~\cite{lemmerling2001misfit}. All prior work focuses on infinite-horizon behaviors. In contrast, we consider the finite-horizon setting, which is increasingly relevant in system identification and data-driven control~\cite{markovsky2021behavioral}. Closest to our work are~\cite{padoan2022behavioral,fazzi2023distance}, but the distances proposed in these works either assign a trivial maximal value or become undefined when the subspace dimensions differ.

\textbf{Contributions}: We introduce a new class of distances for finite-horizon linear behaviors (Theorem~\ref{thm:metrics_restricted_behaviors_any_dimension}), based on principal angles over the doubly infinite Grassmannian~\cite{ye2016schubert}. Such distances are able to compare subspaces of differing dimensions, remain invariant under coordinate changes, rotations, and permutations (Lemma~\ref{lemma:basis-invariance} and Theorem~\ref{thm:rotation-invariance}), and are thus independent of input-output partitions. They naturally encode a trade-off between complexity and misfit (Theorem~\ref{thm:angles-complexity}), yielding utility functions (Lemma~\ref{lemma:utility_function_property}) resolving a long-standing modeling question posed in~\cite{willems1987timeIII}. We show that the \gls{MPUM}~\cite{willems1986timeII} emerges as the optimal model among all unfalsified behaviors (Theorem~\ref{thm:optimization_equivalence}, Corollary~\ref{cor:optimization_model_class_unfalsified models}). Finally, we illustrate the value of these metrics in the context of anomaly detection, where their finer resolution offers improved sensitivity to anomaly severity over existing distances.

\textbf{Organization}: 
Section~\ref{sec:preliminaries} covers preliminary material.
Section~\ref{sec:puzzle} revisits a modeling puzzle. Section~\ref{sec:L_fin} defines the set of finite-horizon linear behaviors. Section~\ref{sec:doubly-infinite-Grassmannian} reviews the doubly infinite Grassmannian~\cite{ye2016schubert}. Section~\ref{sec:distances} characterizes the proposed metrics and their relevance to modeling. Section~\ref{sec:application} illustrates their use in time series anomaly detection.
Section~\ref{sec:conclusion} provides a summary and future research directions.

\section{Preliminaries} \label{sec:preliminaries} 

\subsection{Notation}
The set of positive integers is denoted by $\N$.   
The set of real numbers is denoted by $\R$. 
For ${T\in\N}$, the set of integers $\{1, 2, \dots , T\}$ is denoted by $\mathbf{T}$.  
A map $f$ from $X$ to $Y$ is denoted by ${f:X \to Y}$; $(Y)^{X}$ denotes the set of all such maps. The restriction of ${f:X \to Y}$ to a set ${X^{\prime}}$, with ${X^{\prime} \cap X \neq \emptyset}$, is defined as ${f|_{X^{\prime}}(x) = f(x)}$ for ${x \in X^{\prime}}$; if ${\mathcal{F} \subseteq (Y)^{X}}$, then ${\mathcal{F}|_{X^{\prime}} =\{ f|_{X^{\prime}} \, : \, f \in \mathcal{F}\}}$.
The shift operator ${\sigma}$ on sequences is defined as ${(\sigma w)_t= w_{t+1}}$; if ${\mathcal{S}}$ is a set of sequences,
then ${\sigma \mathcal{S} = \{\sigma w \, : \,  w\in\mathcal{S} \}}$. The transpose, image, and kernel of a matrix ${M\in\R^{m\times n}}$ are denoted by $M^{\top}$, $\Image\, M$, and $\Ker M$, respectively. The $m \times p$ zero matrix is denoted by $0_{m \times p}$ and the $m \times p$ matrix with ones on the main diagonal and zeros elsewhere is denoted by $I_{m \times p}$. For $m = p$, we use $0_p$ and $I_p$, omitting subscripts if clear  from the context.  Following~\cite{ye2016schubert}, the Grassmannian manifold of $k$-dimensional subspaces in $\mathbb{R}^N$ is denoted by $\Grass{k}{N}$, the {orthogonal group} of $k \times k$ orthogonal matrices is denoted by $\Ort{k}$, and the $i$-th principal angle between ${\mathcal{V} \in \Grass{k}{N}}$ and ${\mathcal{U} \in \Grass{l}{N}}$ is denoted by ${\theta_{i}(\mathcal{V}, \mathcal{U})}$ for ${i \in \mathbf{r}}$, with ${r = \min(k,l)}$. For brevity, we omit dependence on $\mathcal{V}$ and $\mathcal{U}$ if clear from the context. 

\subsection{Behavioral system theory}

Behavioral systems theory views systems as sets of trajectories~\cite{willems1986timeI,willems1986timeII,willems1987timeIII}. Given a \textit{time set} ${\T}$ and a  \textit{signal space} $\W$, a 
\textit{trajectory} is any function ${w: \T \to \W}$ and
a \textit{system} is defined by its \textit{behavior} $\B$, \textit{i.e.}, a subset of the set of all possible trajectories  ${\mathcal{W} = (\W)^{\T}}$.   In this spirit, a \textit{model} of a system is its behavior, and a \textit{model class}  ${\mathcal{M}\subseteq 2^{\mathcal{W}}}$  is a family of subsets of the set of possible trajectories  $ {\mathcal{W}} $. By convention, we identify systems with their behaviors and use  ``model,'' ``system,'' and ``behavior'' interchangeably.
Throughout this work, we exclusively focus on \textit{discrete-time} systems, with ${\T = \N}$ and ${\W = \R^q}$,  and adapt definitions accordingly.
A system $\B$ is \textit{linear} if $\B$ is a linear subspace and \textit{time-invariant} if $\B$ is shift-invariant, \textit{i.e.}, ${\sigma \B \subseteq \B}$.  A linear system $\B$ is \textit{complete} if $\B$ is closed in the topology of pointwise convergence~\cite{willems1986timeI}.  The model classes of all such discrete-time, complete, linear and \gls{LTI} systems are denoted by $\L^q$ and $\LTI^q,$ respectively.

\section{A motivating modeling puzzle} \label{sec:puzzle}

Given a \textit{data set} $\mathcal{D} \subseteq \mathcal{W}$ and a model class $\mathcal{M}\subseteq 2^{\mathcal{W}}$, the \textit{modeling} problem is to find a model ${\B \in \mathcal{M}}$ that ``best'' fits the data $\mathcal{D}$ according to one or more objectives~\cite{willems1986timeI,willems1986timeII,willems1987timeIII}. Basic objectives in a modeling problem are 
\begin{itemize}
\item ``simplicity,'' measured by a \textit{complexity} function $c(\B)$, and 
\item ``accuracy,'' measured by a \textit{misfit} function $\epsilon(\mathcal{D},\B).$
\end{itemize}
The complexity function $c(\B)$ quantifies the expressivity of a model $\B$. For example, the \textit{complexity} ${c:  \LTI^q   \to [0,1]^{\N}}$ of a model ${\B\in\LTI^q}$ is  a sequence  defined as~\cite[p.92]{willems1987timeIII}
\beq \label{eq:complexity}  
c_t(\B) = \dim  \B|_\mathbf{t} /qt , \quad t \in \N .  
\eeq  
\noindent
Intuitively, complexity increases with the number of linearly independent trajectories, as fewer constraints must be satisfied. 
The misfit function $\epsilon(\mathcal{D},\B)$ indicates how far the model $\B$ fails to explain the data set $\mathcal{D}$. 
For example, a misfit function for a model ${\B\in\LTI^q}$ and a data set $\mathcal{D}$ consisting of a single trajectory ${w_d\in \mathcal{W}}$ is~\cite[p.96]{willems1987timeIII}
\beq \nn \label{eq:misfit-vector} 
\epsilon(\mathcal{D}, \B) = \inf_{{w} \in \B } \norm{w_d - {w}},
\eeq 
where $\norm{\,\cdot\,}$ is a user-defined norm. Note that complexity and misfit functions are not rigidly specified by a unique definition; rather, their choice depends on the specific use of the model (e.g., estimation, prediction, or control). In fact, we will later  consider a different misfit function and demonstrate its relevance in   measuring distances between \textit{finite-horizon} \gls{LTI} behaviors (see~Section~\ref{sec:distances}).

\subsection{How to balance the complexity–misfit trade-off?}
Large complexity and misfit are both undesirable  and generally competing objectives. 
The complexity–misfit trade-off is thus a fundamental issue in modeling, which naturally leads to the multi-objective optimization problem
\beq \label{eq:modeling} 
\begin{array}{ll}
\underset{\B \in \mathcal{M}}{\mbox{minimize}} & \bma
\begin{array}{c}
c(\B)\\
\epsilon(\mathcal{D}, \B)
\end{array}
\ema 
\end{array} .
\eeq 

\noindent 
J. C. Willems proposes two distinct approaches to characterize optimal solutions of~\eqref{eq:modeling} for \gls{LTI} systems in~\cite{willems1987timeIII}. 
The first approach restricts the model class $\mathcal{M}$ to those \gls{LTI}  systems whose complexity does not exceed a fixed maximal complexity $c^{\textup{adm}}$,  defined as
\beq \nn \label{eq:model-class-maximal-admissible-complexity} 
\mathcal{M}_{\textup{adm}} =\left\{\B\in \LTI^q \,|\, c(\B) \le c^{\textup{adm}}  \right\},
\eeq 
with complexities partially ordered by pointwise domination: 
\beq \nn \label{eq:partial_order_complexities}
c(\B) \leq c(\B^{\prime}) \iff c_t(\B) \leq c_t(\B^{\prime}), \quad \forall \, t \in \N. 
\eeq 
\noindent
A model $\B$ is then optimal if it has the smallest complexity among all models minimizing the misfit within $\mathcal{M}^{\textup{adm}}$.
This echoes the classical system identification approach, where a model structure is fixed \textit{a priori} before parameter estimation. 
The second approach, in contrast, limits the model class $\mathcal{M}$ to those \gls{LTI}  systems achieving at most a given misfit $\epsilon^{\textup{tol}}$, i.e.,
\beq \nn \label{eq:model-class-maximal-tolerated-misfit} 
\mathcal{M}_{\textup{tol}} =\left\{\B\in \LTI^q \,|\, \epsilon(\mathcal{D}, \B) \le \epsilon^{\textup{tol}}  \right\}.
\eeq 
\noindent
A model $\B$ is then optimal if it has the smallest misfit among all models 
minimizing the complexity within $\mathcal{M}^{\textup{tol}}$.
Unlike the first approach, this method allows the complexity of the model to be adjusted to achieve a target misfit. However, finding an appropriate $\epsilon^{\textup{tol}}$ can be challenging in practice~\cite{willems1987timeIII}.

Both modeling approaches sidestep a fundamental question:
``Do there exist utility functions that naturally encode the trade-off between complexity and misfit?''
J. C. Willems himself emphasizes this difficulty in~\cite[p.89]{willems1987timeIII}: ``Most appealing of all is to have a methodology in which a combination of the complexity and the misfit is used in a utility function ... to be maximized. However, ... it seems difficult to come up with an intuitively justifiable utility function.''

In Section~\ref{sec:distances}, we propose a framework aimed at addressing this question by defining a class of ``natural'' utility functions that depend explicitly on system complexity and misfit, providing a clear geometric interpretation of the modeling problem as a minimum distance problem.

\section{Finite-horizon linear behaviors} \label{sec:L_fin}

We define the set of all finite-horizon linear behaviors as  
\beq \nn 
\Lfh = \left\{\, \B|_\mathbb{I} \mid \B \in \L^{q}, \ q \in \N, \ \mathbb{I}\subseteq\N, \ 0<|\mathbb{I}|< \infty\right\}\!,
\eeq
with $|\mathbb{I}|$ the cardinality of the set $\mathbb{I}$. For $q \in \N$, we also define 
$$ \Lfh^q = \left\{\, \B|_\mathbb{I} \mid \B \in \L^{q},  \ \mathbb{I}\subseteq\N, \ 0<|\mathbb{I}|< \infty\right\}\!. $$
The set $\Lfh$ contains all finite-horizon restrictions of complete linear systems, with any number of variables.  
The subset $\Lfh^q$ contains those with exactly $q$ variables.
The restriction operation maps linear behaviors to finite-dimensional subspaces.

\begin{fact} \label{fact:subspace}  
Every element of \(\Lfh\) is a finite-dimensional subspace.  
\end{fact}  

An important class of finite-horizon linear behaviors arises from time-invariant systems. Shift invariance allows one to consider restrictions to time intervals that depend only on the horizon length, without loss of generality. Accordingly, the set of all finite-horizon \gls{LTI} behaviors is defined as
\beq \nn
\LTIfh = \left\{\, \B|_\mathbf{L} \mid \B \in \LTI^{q},  \ q \in \N , \ L \in \N \, \right\},
\eeq  
and the subset of all finite-horizon \gls{LTI} behaviors with ${q \in \N}$ variables is defined as
\beq \nn
\LTIfh^q = \left\{\, \B|_\mathbf{L} \mid \B \in \LTI^{q},   \ L \in \N \, \right\}.
\eeq
Shift invariance determines the dimension of the elements of $\LTIfh^q$ in terms of integer invariants and the horizon length.
In fact, every system ${\B \in \LTI^{q}}$ admits a \textit{kernel representation},
$$\B = \ker R(\sigma) ,$$
where  $R(\sigma)$ is the operator defined by the polynomial matrix 
$R(z) = R_0 +R_1 z +\ldots+ R_{\ell}z^\ell,$
 with $\ell\in\N$, ${R_i\in\R^{p\times q}}$ 
for ${i\in\text{\boldmath$\ell$\unboldmath},}$ and 
$\Ker R(\sigma) = \{w\in\mathcal{W} : R(\sigma)w = 0\}.$
Without loss of generality, one may assume $\Ker R(\sigma)$  is a \emph{minimal}  kernel representation of $\B$,  \textit{i.e.}, $R(\sigma)$ has full row rank~\cite{willems1986timeI}.

The structure of a system ${\B\in\LTI^q}$ is characterized by its \textit{integer invariants}~\cite{willems1986timeI}. 
Among all such invariants, the most relevant for our purposes are defined as
\begin{itemize}
\item the \textit{number of inputs} ${m(\B) = q-\text{row dim} \, R(\sigma)}$, 
\item the \textit{lag} ${\ell(\B) = \max_{i}\{\deg\text{row}_i \, R(\sigma) \}}$, and 
\item the \textit{order} ${n(\B)=\sum_{i} \deg\text{row}_i \, R(\sigma) }$, 
\end{itemize}
where   
$\Ker R(\sigma)$  is a minimal  kernel representation of $\B$, while   
${\text{row dim}R(\sigma)}$ and ${\deg\text{row}_i R(\sigma)}$ are the number of rows and the degree of the $i$-th row of $R(\sigma)$, respectively. The integer invariants are intrinsic properties of an \gls{LTI} system, as they do not depend on its representation~\cite{willems1986timeI}. 

\begin{lemma}\cite{markovsky2021behavioral} \label{lemma:subspace} 
Let ${\B \in \LTI^{q}}$. For ${L\in\N}$,  with ${L\ge \ell(\B)}$, $\B|_\mathbf{L}$ is a subspace of dimension ${\dim \B|_\mathbf{L}  =  m(\B) L+ n(\B)}$. 
\end{lemma}
\noindent
By Lemma~\ref{lemma:subspace}, finite-horizon \gls{LTI} behaviors can be represented  by raw data matrices. A version of this principle is known in the literature as the \textit{Willems' fundamental lemma}~\cite{markovsky2021behavioral}. 

Fact~\ref{fact:subspace} and Lemma~\ref{lemma:subspace} show that both \( \Lfh \) and \( \LTIfh \) consist of finite-dimensional subspaces. As such, their elements can be naturally viewed as points on the \emph{doubly infinite Grassmannian} \( \mathrm{Gr}(\infty, \infty) \), which collects all finite-dimensional subspaces into a common ambient space, enabling comparisons across different dimensions and horizons. The use of this space provides a convenient framework in which dimension-agnostic metrics—such as those introduced in~\cite{ye2016schubert}—can be applied to finite-horizon behaviors. Section~\ref{sec:doubly-infinite-Grassmannian} provides a brief overview of this space and the metrics it supports, which form the basis of our proposed framework.

\section{The doubly infinite Grassmannian} \label{sec:doubly-infinite-Grassmannian}

The doubly infinite Grassmannian is the set of all finite-dimensional subspaces, each embedded in a common ambient space. We recall its formal definition from~\cite[p. 1187]{ye2016schubert}.

\begin{definition}
 The doubly infinite Grassmannian $\Grass{\infty}{\infty}$ is the direct limit of the direct system of Grassmannians  $\{ \Grass{k}{N} \,|\, (k,N) \in \N\times\N \},$ with canonical inclusion maps\footnote{By convention, if $(M-N-l+k)$ or $l-k$ are zero, the rows or columns corresponding to ${0}_{(M-N-l+k)\times k}$ and $I_{l-k}$ are omitted.} 
 \beq  \label{eq:inclusion_maps}
 \!\!\!
 \begin{aligned}[c]
 \iota_{NM}^{kl} \colon \Grass{k}{N} &\longrightarrow  \Grass{l}{M}  \\
 						  \Image V &\longmapsto 
 \Image 
 \bma \scalebox{0.9}{$
 \begin{array}{cc}
 V & {0} \\
 {0}_{(M-N-l+k) \times k} & {0} \\
 {0} & I_{l-k}
 \end{array} $}
 \ema , 
 \end{aligned} 
 \eeq
 for all ${k\in \mathbf{l}}$ and ${N \in \mathbf{M}  }$ such that ${l-k\le M-N}$. 
\end{definition}
\noindent
The family of inclusion maps~\eqref{eq:inclusion_maps} identifies each element of $\Grass{k}{N}$ with its image in $\Grass{l}{M}$,
treating a subspace \( \mathcal{V} \) and its image \( \iota_{NM}^{kl}(\mathcal{V}) \) as the same object. 
Fig.~\ref{fig:injections} illustrates this point for ${k = 1}$, ${N = 2}$, ${l = 2}$, and ${M = 3}$.
\begin{figure}[h!]
    \centering
\begin{tikzpicture}[black,scale=.95, line cap=round, line join=round, >=latex, thick]
    \begin{scope}[shift={(-2,0)}]
        
    \draw[black] (0,0) circle (1);
    
    \draw[thick, -latex] (-1.5,0) -- (1.5,0) node[right] {${x}$};
    \draw[thick, -latex] (0,-1.5) -- (0,1.5) node[above] {${y}$};

    \node[below right] at (1,0) {${1}$};

    \draw[thick, cb-orange, -latex] (0,0) -- (60:1.0008) node[anchor=north east, xshift=17.5pt, yshift=9.5pt] {${V}$};

    \draw[thick, dashed, cb-orange] (-0.85,-1.5) -- (60:1.75) node[above right] {${\mathcal{V}}$};
    
    \end{scope}


    \draw[very thick, -latex] (.25,0) -- (1.25,0)
        node[midway, above] {\large $\iota_{NM}^{kl}$};


    \begin{scope}[shift={(3.5,0)}, x={(1cm,0cm)}, y={(0.5cm,0.5cm)}, z={(0cm,1cm)}]

        \def\x{0.5}    
        \def\y{0.866}  
        \def\X{1.75*\x} 
        \def\Y{1.75*\y} 
        \def\z{1}      

        \fill[cb-purple, opacity=0.2] (0,-1.5,-1.5) -- (0,1.5,-1.5) -- (0,1.5,1.5) -- (0,-1.5,1.5) -- cycle;

        \draw[thick, -latex] (-1.5,0,0) -- (1.5,0,0) node[right] {${x}$};
        \draw[thick, -latex] (0,-1.5,0) -- (0,1.5,0) node[above] {${y}$};
        \draw[thick, -latex] (0,0,-1.) -- (0,0,1.5) node[left] {${z}$};

        \draw[thick] (0,0,0) circle (1);

        \node[below] at (1,0,0) {${1}$};

        \foreach \t in {240,245,...,265} {
            \draw[very thick, cb-green] ({cos(\t)},{sin(\t)},0) -- ({cos(\t+5)},{sin(\t+5)},0);
        }

        \node[cb-green, right] at ({1.75*cos(255)},{1.75*sin(255)},0) {$\theta_1$};

        \draw[-latex, cb-orange] (0,0,0) -- (\x,\y,0) node[above] {$V$};

        \draw[dashed, cb-orange] (0,0,0) -- (-\X,-\Y,0);
        \draw[dashed, cb-orange] (0,0,0) -- (\X,\Y,0) node[above right] {$\mathcal{V}$};

        \node[cb-orange, anchor=north east, xshift=15pt, yshift=15pt] at (\X,\Y,\z) {$\iota_{NM}^{kl}(\mathcal{V})$};
        \node[black, below] at (0,0,-1.5) {${\mathcal{U}}$};

        \fill[cb-orange, opacity=0.1] (-\X,-\Y,0) -- (-\X,-\Y,\z) -- (\X,\Y,\z) -- (\X,\Y,0) -- cycle;
        \fill[cb-orange, opacity=0.1] (-\X,-\Y,0) -- (-\X,-\Y,-\z) -- (\X,\Y,-\z) -- (\X,\Y,0) -- cycle;
        
    \end{scope}

\end{tikzpicture}
\caption{The map $ \iota_{NM}^{kl} $ embeds $ \Grass{k}{N} $ into $ \Grass{l}{M} $, as illustrated for ${k = 1}$, ${N = 2}$, ${l = 2}$, and ${M = 3}$. The subspace ${\mathcal{V} = \Image (V)}$ in $ \R^2 $ \textcolor{cb-orange}{(dashed, orange)} is mapped to the corresponding subspace in ${\R^3}$ \textcolor{cb-orange}{(shaded, orange)}. The principal angle $\theta_1$ \textcolor{cb-green}{(solid, green)}  measures the distance to the subspace ${\mathcal{U} = \Image  [ \,0  \;\; I\,]^\top \in  \R^3}$ \textcolor{cb-purple}{(shaded, purple)}.}
    \label{fig:injections}
\end{figure}
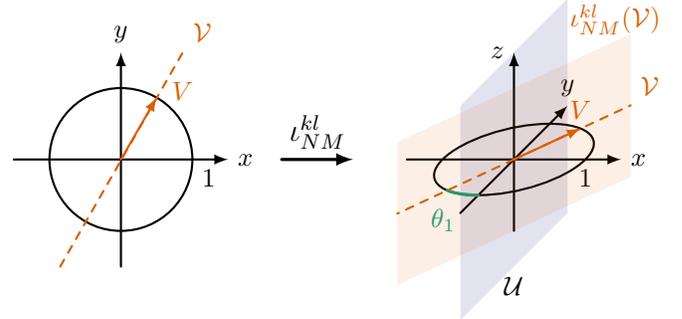

A number of standard Grassmannian metrics can be extended to the doubly infinite Grassmannian~\cite{ye2016schubert}. 
While several metrics assign maximal distance to subspaces of different dimensions, others 
are able to capture dimension mismatches, offering a finer resolution.   Table~\ref{tab:distances} recalls three such metrics: the \textit{chordal}, \textit{Grassmann}, and \textit{Procrustes} metrics.

\begin{table}[h!]
\centering
\caption{Metrics between subspaces $\mathcal{V}$ and $\mathcal{U}$, with ${\dim \mathcal{V} = k}$ and ${\dim \mathcal{U} = l}$, in terms of the principal angles $\theta_i(\mathcal{V},\mathcal{U})$.}
\begin{tabular}{ll} \hline
Name	& Metric 										 \\[0.3em] \hline
Chordal			
& $d_\kappa(\mathcal{V},\mathcal{U})	=\left(|k -l|+\displaystyle\sum_{i=1}^{\min(k,l)} (\sin \theta_i)^2 \right)^{1/2}$ \\
Grassmann		
& $d_\gamma(\mathcal{V},\mathcal{U})	=\left(|k -l| \tfrac{\pi^2}{4} + \displaystyle\sum_{i=1}^{\min(k,l)} \theta_i^2 \right)^{1/2}$        \\
Procrustes		
& $d_\pi(\mathcal{V},\mathcal{U})	=\left(|k -l|+\displaystyle2\sum_{i=1}^{\min(k,l)} \left(\sin \tfrac{\theta_i}{2}\right)^2 \right)^{1/2}$\\[0.3em] \hline
\end{tabular}
\vspace{0.1cm}
\label{tab:distances}
\end{table}%

\begin{proposition} \cite{ye2016schubert} \label{prop:metrics}
Every expression in Table~\ref{tab:distances} defines a metric on~$\Grass{\infty}{\infty}$. \end{proposition}

\noindent
All metrics in Table~\ref{tab:distances} depend on the principal angles but differ in how they aggregate them: chordal uses the $2$-norm of their sines, Grassmann the $2$-norm of the angles themselves, and Procrustes the Frobenius norm between orthogonal projectors. They emphasize different aspects, namely, numerical robustness, geodesic distance, and alignment error, respectively.

\begin{remark}
    All metrics in Table~\ref{tab:distances} can be expressed as
    \begin{equation} \label{eq:general_distance}
        d_*(\mathcal{V}, \mathcal{U}) = \left(\delta_*(\mathcal{V}, \mathcal{U})^2 + \alpha_*^2 |\dim \mathcal{V} - \dim \mathcal{U}|\right)^{1/2},
    \end{equation}
    where  $*\in\{ \kappa, \gamma, \pi\}$,  ${\alpha_* \in \R}$ is defined as
    \begin{equation} \label{eq:coefficients}
         \alpha_\kappa = 1, \quad \alpha_\gamma = \frac{\pi}{2},  \quad \alpha_\pi = 1,
    \end{equation}
    and $\delta_*(\mathcal{V}, \mathcal{U})$ is defined as
    \bseq \label{eq:premetrics}
    \begin{align}
    \delta_\kappa(\mathcal{V},\mathcal{U}) &= \scalebox{1}{$\left(\sum_{i=1}^{\min(k,l)} (\sin \theta_i)^2 \right)^{1/2}$}, \\
    \delta_\gamma(\mathcal{V},\mathcal{U}) &= \scalebox{1}{$\left(\sum_{i=1}^{\min(k,l)} \theta_i^2 \right)^{1/2}$}, \\
    \delta_\pi(\mathcal{V},\mathcal{U}) &= \scalebox{1}{$\left( 2\sum_{i=1}^{\min(k,l)} \left(\sin \frac{\theta_i}{2}\right)^2 \right)^{1/2}$}.
    \end{align}
    \eseq
    \noindent
    All metrics can be efficiently evaluated, as the ranks and principal angles can be computed via QR decomposition and \gls{SVD}~\cite{ye2016schubert}, with a worst-case complexity of $\mathcal{O}(qT \cdot (k^2 + l^2) + kl \min\{k, l\})$.
\end{remark}

\section{Distances for finite-horizon linear behaviors}
\label{sec:distances}

We now define a class of metrics on the set of all finite-horizon linear behaviors $\Lfh$, inherited from the geometry of the doubly infinite Grassmannian $\Grass{\infty}{\infty}$. 

\begin{theorem}[Induced metrics on $\Lfh$] \label{thm:metrics_restricted_behaviors_any_dimension} Every metric on $\Grass{\infty}{\infty}$ defines a metric on~$\Lfh$. In particular, every metric in Table~\ref{tab:distances} defines a metric on~$\Lfh$.
\end{theorem}
  
\begin{proof} 
The argument is inspired by~\cite[Proposition 1]{padoan2022behavioral}.
By Fact~\ref{fact:subspace}, each element of $\Lfh$ is a finite-dimensional subspace, making $\Lfh$ a subset of $\Grass{\infty}{\infty}$. Since restricting a metric to a subset yields a well-defined metric on that subset, any metric on $\Grass{\infty}{\infty}$ defines a metric on $\Lfh$.
\end{proof}

\noindent
Table~\ref{tab:distances} presents metrics that exhibit properties relevant to the analysis of finite-horizon linear behaviors, as discussed next.

\begin{lemma}[Coordinate change invariance] \label{lemma:basis-invariance} 
Let ${H \in \mathbb{R}^{m \times n}}$ and ${H^{\prime} \in \mathbb{R}^{m^{\prime} \times n^{\prime}}}$ be non-zero matrices.
Let $d$ be a metric in Table~\ref{tab:distances}. Then 
${d (\Image(H  P) , \Image(H^{\prime} P^{\prime})) = d (\Image H , \Image H^{\prime})  }$
for all invertible ${P \in \R^{n \times n }}$ and ${P^{\prime} \in \R^{n^{\prime} \times n^{\prime} }}$. 
\end{lemma}
\noindent

\noindent
Lemma~\ref{lemma:basis-invariance} follows from the identity \( {\Image(H P) = \Image(H)} \) for any invertible matrix \( {P \in \mathbb{R}^{n \times n}} \), where \( {H \in \mathbb{R}^{m \times n}} \) is a matrix representative of a behavior in \( \Lfh \). It formalizes that the subspace \( \Image(H) \), and thus the behavior it represents, is invariant under changes of basis in its column space. Thus, the distance between two behaviors in \( \Lfh \) is invariant under such coordinate transformations. The same invariance holds for all metrics listed in Table~\ref{tab:distances}, each of which depends only on the subspaces and is thus well-defined on equivalence classes of matrix representations that define the same subspaces.

\begin{remark}[Limitations of matrix distances]\label{rem:limitations_matrix_metrics}
Let \({H \in \mathbb{R}^{m \times n}}\) and \({H' \in \mathbb{R}^{2 \times 1}}\), with \({H = \epsilon I_{m\times n}}\) and  \( {H' = [0\;\;\epsilon]^\top  }\)  for some \({\epsilon > 0}\). These matrices can only be compared using distances induced by matrix norms (e.g., the Frobenius norm) when they have the same dimensions (${m = 2}$ and ${n= 1}$), in which case they yield a distance proportional to \( \epsilon \), which can be arbitrarily small. In contrast, any metric in Table~\ref{tab:distances} assigns maximal distance between the subspaces spanned by their columns. This highlights a key advantage of the metrics in Table~\ref{tab:distances} — they capture dimension changes and preserve geometric structure, while distances induced by matrix norms depend on representation and may assign arbitrarily small values to geometrically distant objects. \end{remark} 

\noindent
A natural requirement for distances on subspaces is that they depend only on their relative positions, remaining invariant under global rotations of the ambient space. Formally, a distance ${d}$ on ${\Grass{k}{N}}$ is said to be \textit{rotationally invariant} if
$
{d(Q \cdot \mathcal{V}, Q \cdot \mathcal{U}) = d(\mathcal{V}, \mathcal{U})}
$
for all ${Q \in \Ort{N}}$ and ${\mathcal{V}, \mathcal{U} \in \Grass{k}{N}}$,  where the (left) group action of $\Ort{N}$ on $\Grass{k}{N}$ is defined as
${Q \cdot \mathcal{V} = \Image(QV),}$
with ${V \in \mathbb{R}^{N \times k}}$ such that ${V^{\transpose} V = I}$ and ${\mathcal{V} = \Image(V)}$.
As shown in~\cite[Theorem~2]{ye2016schubert}, any distance satisfying this invariance property must be a function of the principal angles between ${\mathcal{V}}$ and ${\mathcal{U}}$. Next, we establish an analogous property for finite-horizon linear behaviors of different dimensions.

\begin{theorem}[Rotational invariance] \label{thm:rotation-invariance} 
Let $\B|_{\mathbb{I}}, \B^{\prime}|_{\mathbb{I}} \in \Lfh^q$ and let \( d \) be one of the metrics in Table~\ref{tab:distances}. Then 
$$d(Q \cdot {\B}|_{\mathbb{I}}, Q \cdot \B^{\prime}|_{\mathbb{I}}) = d({\B}|_{\mathbb{I}}, \B^{\prime}|_{\mathbb{I}}),$$
for all $Q \in O(q|\mathbb{I}|)$. 
\end{theorem}

\begin{proof}
Fix $N=q|\mathbb{I}|$. By Fact~\ref{fact:subspace}, 
there exist $k,l\in\N$, ${\mathcal{V} \in \Grass{k}{N}}$, and ${\mathcal{U} \in \Grass{l}{N}}$ such that $\mathcal{V} = {\B}|_\mathbb{I}$ and $\mathcal{U} =\B^{\prime}|_\mathbb{I}$.  Let ${r = \min(k, l)}$. 
By  definition of principal angles,  ${\theta_i(Q \cdot \mathcal{V}, Q \cdot  \mathcal{U}) = \theta_i(\mathcal{V}, \mathcal{U})}$ for all ${Q \in O(N)}$ and ${i \in \mathbf{r}}$.  Furthermore, ${\dim Q \cdot \mathcal{V} = \dim\mathcal{V}}$  and ${\dim Q \cdot  \mathcal{U} = \dim \mathcal{U}}$. Thus,~\eqref{eq:general_distance} directly implies the desired identity.
\end{proof}

\begin{remark}[Permutation invariance] 
    Theorem~\ref{thm:rotation-invariance} implies 
    $$d(\Pi \cdot {\B}|_{\mathbb{I}}, \Pi \cdot \B^{\prime}|_{\mathbb{I}}) = d({\B}|_{\mathbb{I}}, \B^{\prime}|_{\mathbb{I}}),$$
    for all ${\Pi \in \text{Per}(q|\mathbb{I}|)}$, where ${\text{Per}(n)}$ is the set of all $n\times n$ permutation matrices. Thus, every metric in Table~\ref{tab:distances} does not depend on specific input-output partitions~\cite[p.568]{willems1986timeI}, making the metric suitable if such partitions are not defined \textit{a priori}.
\end{remark}

\noindent
Shift invariance imposes additional structure, allowing one to derive stronger properties for the metrics in Table~\ref{tab:distances} when applied to finite-horizon \gls{LTI} behaviors.

\begin{theorem}[Dependence on principal angles and complexity] \label{thm:angles-complexity} 
Let $\B|_{\mathbf{L}}, \B^{\prime}|_{\mathbf{L}} \in \LTIfh^q$, with ${L\in\N}$, and let \( d_* \)  be one of the metrics in Table~\ref{tab:distances},  with \( * \in \{ \kappa, \gamma, \pi \} \). 
 Then 
\begin{equation} \nn
d_*(\B|_{\mathbf{L}}, \B^{\prime}|_{\mathbf{L}})^2 =  
        \delta_*(\B|_{\mathbf{L}}, \B^{\prime}|_{\mathbf{L}})^2 + 
        \alpha_*^{2} q L \left|  c_L(\B) - c_{{L}}(\B^{\prime}) \right| . \label{eq:metric-dependence-angles-complexity}
\end{equation} 
\end{theorem}

\noindent
The claim follows directly from~\eqref{eq:complexity} and~\eqref{eq:general_distance}.

\begin{remark}[Dependence on integer invariants]
Theorem~\ref{thm:angles-complexity}  can be refined to explicitly depend on integer invariants (e.g., order) under some assumptions. For example, in view of Lemma~\ref{lemma:subspace},  if ${m(\B) = m(\B^{\prime})}$  and ${L \ge \max(\ell(\B), \ell(\B^{\prime}))}$, then  
${d_{*}(\B|_{\mathbf{L}}, \B^{\prime}|_{\mathbf{L}})^2 = 
        \delta_*(\B|_{\mathbf{L}}, \B^{\prime}|_{\mathbf{L}})^2 + \alpha_*^2 |n(\B)-n(\B^{\prime})|.}$
Thus, \( d_* \) explicitly depends (linearly) only on the order mismatch, provided the time horizon is sufficiently long.
\end{remark}

\subsection{Modeling as a minimum distance problem} \label{ssec:application}

We now show that the proposed metrics  define natural utility functions that explicitly depend on both system complexity and misfit, thus addressing the modeling puzzle in Section~\ref{sec:puzzle}. To this end, we recall the notion of \gls{MPUM} first introduced in~\cite[Definition 4]{willems1986timeII}.

\begin{definition}[Most powerful unfalsified model] ~\cite{willems1986timeII} \label{def:MPUM}
Given a model class $\mathcal{M}\subseteq 2^{\mathcal{W}}$ and a data set ${ \mathcal{D} \subseteq \mathcal{W}}$, the \textit{most powerful unfalsified model ${\B_{\textup{mpum}}(\mathcal{D})}$ in the model class $\mathcal{M}$ based on the data set $\mathcal{D}$} is defined by three properties: (i) ${\B_{\textup{mpum}}(\mathcal{D}) \in \M}$,
(ii) ${\mathcal{D} \subseteq \B_{\textup{mpum}}(\mathcal{D})}$, (iii) and  ${\B \in \M}$ and ${\mathcal{D} \subseteq \B}$ imply ${\B_{\textup{mpum}}(\mathcal{D}) \subseteq \B}$.
\end{definition}
\noindent
The \gls{MPUM} may not exist, but when it exists it is unique~\cite{willems1986timeII}.
The \gls{MPUM} in the model class $\LTI^q$ always exists and is~\cite{willems1986timeII} 
\beq \nn \label{eq:MPUM_linear}
\B_{\text{mpum}}(\mathcal{D}) = \overline{\Span}\{\mathcal{D}, \sigma \mathcal{D}, \sigma^2 \mathcal{D}, \ldots \}, 
\eeq 
where $\overline{\Span}(\mathcal{S})$ denotes the closure of the span of the set $\mathcal{S}$ with respect to the topology of pointwise convergence.

Next, we define misfit and utility functions that are instrumental to characterize modeling as a geometric problem.  Given a model ${\mathcal{B} \in \LTI^q}$, a data set ${\mathcal{D} \subseteq \mathcal{W}}$, a time horizon ${L\in\N}$, 
and an index ${*\in\{ \kappa, \gamma, \pi\}}$, define the \textit{misfit function}
\beq \label{eq:misfit_function_fin}
    \epsilon_{*,L}(\mathcal{D}, \B) = \delta_{*}\left(\B_{\text{mpum}}(\mathcal{D})|_{\mathbf{L}}, \B|_{\mathbf{L}} \right)^{2}, 
\eeq
with ${\delta_*}$ defined as in~\eqref{eq:premetrics}, and the \textit{utility function}
\begin{multline} 
\mu_{*,L}(\mathcal{D},\B) = - \delta_{*} \left( \B_{\text{mpum}}(\mathcal{D})|_{\mathbf{L}}, \B|_{\mathbf{L}} \right)^{2}  \\
-\alpha_*^{2} q L \left| c_L( \B_{\text{mpum}}(\mathcal{D})) - c_L(\B) \right|, \label{eq:utility_function_fin}
\end{multline}
with $c_{L}$ and ${\alpha_*}$ defined as in~\eqref{eq:complexity} and~\eqref{eq:coefficients}, respectively.

The misfit function~\eqref{eq:misfit_function_fin} quantifies the mismatch between \( \mathcal{B}_{\text{mpum}}(\mathcal{D}) \) and \( \mathcal{B} \) over a given interval, while the utility function~\eqref{eq:utility_function_fin} trades off model complexity and misfit  discrepancies over the same interval.  The next result formalizes a key connection between these quantities and the metrics in Table~\ref{tab:distances}.

\begin{lemma}[Subspace distances as utility functions] \label{lemma:utility_function_property}
    Consider a model ${ \mathcal{B} \in \LTI^q}$, a data set ${ \mathcal{D} \subseteq \mathcal{W} }$,  a time horizon  ${ L \in \mathbb{N} }$, and an index ${ * \in \{ \kappa, \gamma, \pi \} }$. Then  
    $$\mu_{*,L}(\mathcal{D},\mathcal{B})  = - d_*\big(\mathcal{B}_{\textup{mpum}}(\mathcal{D})|_{\mathbf{L}}, \mathcal{B}|_{\mathbf{L}}\big)^{2} ,$$
    where \( d_* \) is defined as in Table~\ref{tab:distances}.
\end{lemma}
\noindent

\noindent
The claim follows directly from~\eqref{eq:complexity},~\eqref{eq:general_distance}, and~\eqref{eq:utility_function_fin}.

Lemma~\ref{lemma:utility_function_property} has key implications for the modeling problem, providing a geometric characterization of the  complexity–misfit  trade-off.  Specifically, the optimization problem
\beq \label{eq:opt_problem_1}
\underset{\B \in \mathcal{M}}{\mbox{maximize}} \quad \mu_{*,L}(\mathcal{D},\B)
\eeq
with ${\mathcal{M}\subseteq \LTI^q}$ and $\mu_{*,L}$ defined as in~\eqref{eq:utility_function_fin}, is equivalent to the optimization problem
\begin{equation} \label{eq:opt_problem_2}
\underset{\B \in \mathcal{M}}{\mbox{minimize}} \quad d_*(\B_{\textup{mpum}}(\mathcal{D})|_{\mathbf{L}},\B|_{\mathbf{L}})^{2},
\end{equation}
with $d_*$ defined as in Table~\ref{tab:distances}, in the sense that their feasible sets coincide and their objective functions are identical up to a sign change.
In other words, ``optimal'' modeling according to the given utility function is a minimum distance problem.

\begin{theorem}[Modeling as a minimum distance problem] \label{thm:optimization_equivalence}
Consider a model class ${\mathcal{M} \subseteq \LTI^q}$, a data set ${ \mathcal{D} \subseteq \mathcal{W} }$, a time horizon ${ L \in \mathbb{N} }$, and an index ${ * \in \{ \kappa, \gamma, \pi \} }$. Then the optimization problems~\eqref{eq:opt_problem_1} and~\eqref{eq:opt_problem_2} are equivalent. Consequently, the following properties hold.
\begin{itemize}
    \item[(a)] A model is optimal for one problem if and only if it is optimal for the other.
    \item[(b)] The optimal values differ only in sign:
    \begin{equation} \nn
        \max_{\B \in \mathcal{M}} \mu_{*,L}(\mathcal{D},\B) = - \min_{\B \in \mathcal{M}} d_{*}(\B_{\textup{mpum}}(\mathcal{D})|_{\mathbf{L}}, \B|_{\mathbf{L}})^{2}.
    \end{equation}
    \item[(c)] The set of optimizers is the same:
    \begin{equation}\nn
        \arg\max_{\B \in \mathcal{M}} \mu_{*,L}(\mathcal{D},\B) = \arg\min_{\B \in \mathcal{M}} d_{*}(\B_{\textup{mpum}}(\mathcal{D})|_{\mathbf{L}}, \B|_{\mathbf{L}})^{2}.
    \end{equation}
\end{itemize}
\end{theorem}

\noindent
The role of the utility function becomes particularly interesting when considering the model class of \gls{LTI} systems unfalsified by a given data set \( \mathcal{D} \subseteq \mathcal{W} \), defined as
\beq \label{eq:model_class_compatible_with_data}
\mathcal{M} =\left\{\B\in \LTI^q \,|\, \mathcal{D} \subseteq \B  \right\} .
\eeq
As expected, \( \mathcal{B}_{\text{mpum}}(\mathcal{D}) \) is an optimizer of~\eqref{eq:opt_problem_1} and~\eqref{eq:opt_problem_2}. Moreover, any  optimizer necessarily satisfies  
\beq \label{eq:identity_B_MPUM}
\mathcal{B}|_{\mathbf{L}} = \mathcal{B}_{\textup{mpum}}(\mathcal{D})|_{\mathbf{L}}.
\eeq

\begin{corollary} [Optimality of the \gls{MPUM}]\label{cor:optimization_model_class_unfalsified models}
Consider a data set ${ \mathcal{D} \subseteq \mathcal{W} }$, the model class ${\mathcal{M}}$ defined in~\eqref{eq:model_class_compatible_with_data}, a time horizon ${ L \in \mathbb{N} }$, and an index ${ * \in \{ \kappa, \gamma, \pi \} }$. Then $\B_{\textup{mpum}}(\mathcal{D})$ is an optimizer of~\eqref{eq:opt_problem_1} and~\eqref{eq:opt_problem_2}.
Moreover, any optimizer ${\B\in\mathcal{M}}$ of~\eqref{eq:opt_problem_1} or~\eqref{eq:opt_problem_2} is such that~\eqref{eq:identity_B_MPUM} holds. 
\end{corollary}

\begin{proof} By definition, ${\B_{\textup{mpum}}(\mathcal{D})\in\mathcal{M}}$. Since $d_*$ is non-negative, any ${\B\in\mathcal{M}}$ satisfying  $d_*(\B_{\textup{mpum}}(\mathcal{D})|_{\mathbf{L}},\B|_{\mathbf{L}})=0$ is an optimizer of~\eqref{eq:opt_problem_1}. Therefore, $\B_{\textup{mpum}}(\mathcal{D})$ is an optimizer of~\eqref{eq:opt_problem_1}. By Theorem~\ref{thm:optimization_equivalence}, $\B_{\textup{mpum}}(\mathcal{D})$ is also an optimizer of~\eqref{eq:opt_problem_2}, proving the first statement. Let ${\B\in\mathcal{M}}$ be an optimizer of~\eqref{eq:opt_problem_1} or~\eqref{eq:opt_problem_2}. By Theorem~\ref{thm:optimization_equivalence}, $\B$ must satisfy ${d_*(\B_{\textup{mpum}}(\mathcal{D})|_{\mathbf{L}},\B|_{\mathbf{L}})=0}$.
By Theorem~\ref{thm:angles-complexity}, $c_L(\B_{\textup{mpum}}(\mathcal{D}))=c_L(\B)$ and $\delta_*(\B_{\textup{mpum}}(\mathcal{D})|_{\mathbf{L}},\B|_{\mathbf{L}}) = 0$. Hence, by~\cite[Lemma 13]{ye2016schubert} and by definition of~\gls{MPUM}, ${\B_{\textup{mpum}}(\mathcal{D})|_{\mathbf{L}} \subseteq \B|_{\mathbf{L}}}$. Finally,  ${c_L(\B_{\textup{mpum}}(\mathcal{D})) = c_L(\B)}$ implies ${\dim \B_{\textup{mpum}}(\mathcal{D})|_{\mathbf{L}} = \dim \B|_{\mathbf{L}}}$ and, thus, ${\B_{\textup{mpum}}(\mathcal{D})|_{\mathbf{L}} = \B|_{\mathbf{L}}}$.
\end{proof}

\section{Anomaly detection via subspace distances} \label{sec:application}

Anomaly detection from time-series data is central to many applications~\cite{chandola2009}, including structural health monitoring, seismic analysis, and audio processing. The goal is to identify deviations from nominal behavior that may indicate faults, environmental changes, or unexpected events.
We illustrate the relevance of our metrics in this setting, where their finer resolution yields improved sensitivity over existing distances. Rather than comparing raw signals, we adopt a system-theoretic approach inspired by~\cite{de2002subspace}, evaluating differences between linear behaviors defined by data matrices.

We consider a nominal behavior $\mathcal{B}$ defined by a 0.2~Hz sine wave sampled uniformly over the interval $\mathbb{I} = [0,\, 250]$. The behavior $\mathcal{B}$ is modeled as the output of a second-order autonomous \gls{LTI} system. Anomalies are modeled as additive harmonic perturbations: a 0.1~Hz sine wave is injected for $t \in [50,\, 100]$ (Fault~1), and an additional 0.05~Hz sine wave is injected for $t \in [150,\, 200]$ (Fault~2). Both perturbations induce abrupt deviations from the nominal behavior.

Inspired by the approach outlined in~\cite{padoan2022behavioral}, the signal $y$ is processed online via Hankel matrices $H_t \in \mathbb{R}^{T \times \tau}$, where each column collects ${T = 10}$  output samples, and the matrix includes ${\tau = 16}$  such time-shifted segments. Under the stated assumptions, $\Image(H_t)$ defines a finite-horizon (linear) behavior. The nominal finite-horizon behavior $\mathcal{B}|_{\mathbf{T}}$ is constructed analogously from offline data. At each time $t$, we compute the distance between $\Image(H_t)$ and $\mathcal{B}|_{\mathbf{T}}$ using the chordal metric in Table~\ref{tab:distances} and the $L$-gap metric from~\cite{padoan2022behavioral}. While the chordal metric reflects differences in both geometry and dimension, the $L$-gap metric assigns maximal distance when the subspace dimensions differ, limiting its resolution to anomaly severity. 

Figure~\ref{fig:anomaly_detection} shows that all metrics yield small distances under nominal conditions. During anomalies, however, the signal acquires additional spectral content, increasing the rank of the Hankel matrix $H_t$. The chordal distance captures both the onset and severity of these changes, scaling with the number of added harmonics. Regime switches cause transients during which the behavior defined by the sliding window contains ``mixed'' dynamics. After transients, the behavior defined by the sliding window is linear, and the distance reflects the true extent of deviation.
In contrast, the $L$-gap metric only detects the presence of anomalies, but saturates at its maximum value when subspace dimensions differ, making it insensitive to anomaly severity.  This illustrates the advantage of the proposed metrics: they support dimension-agnostic comparisons and provide graded sensitivity to deviations from nominal behavior.

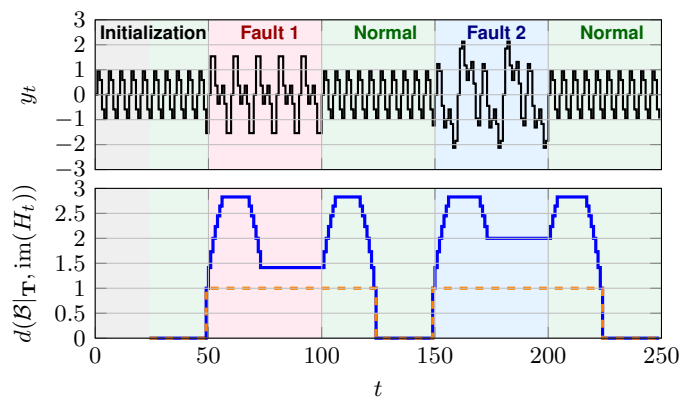
\begin{figure}[t!]
\centering

\definecolor{grayarea}{rgb}{0.85, 0.85, 0.85}
\definecolor{greenarea}{rgb}{0.80, 0.93, 0.85}
\definecolor{redarea}{rgb}{1.00, 0.80, 0.88}
\definecolor{bluearea}{rgb}{0.75, 0.88, 1.00}

\begin{tikzpicture}
\begin{groupplot}[
    group style={
        group size=1 by 2,
        vertical sep=0.25cm
    },
    width=0.85\linewidth,
    height=0.225\linewidth,
    xmin=0, xmax=250,
    xtick distance=50,
    axis on top,
    grid=both,
    scale only axis,
    tick label style={font=\small},
    label style={font=\small},
    title style={font=\small}
]

\nextgroupplot[black,
    ylabel={$y_t$},
    ymin=-3, ymax=3,
    xticklabels=\empty,
    ytick distance=1,
    clip=false
]
\addplot [draw=none, fill=grayarea, opacity=0.4] coordinates {(0, -3) (24, -3) (24, 3) (0, 3)};
\addplot [draw=none, fill=greenarea, opacity=0.4] coordinates {(24, -3)  (24, 3) (50, 3) (50, -3)};
\addplot [draw=none, fill=redarea, opacity=0.4] coordinates {(50, -3) (100, -3) (100, 3) (50, 3)};
\addplot [draw=none, fill=greenarea, opacity=0.4] coordinates {(100, -3) (150, -3) (150, 3) (100, 3)};
\addplot [draw=none, fill=bluearea, opacity=0.4] coordinates {(150, -3) (200, -3) (200, 3) (150, 3)};
\addplot [draw=none, fill=greenarea, opacity=0.4] coordinates {(200, -3) (250, -3) (250, 3) (200, 3)};
\addplot [const plot, thick, black] table [x index=0, y index=1, col sep=comma] {output_signal.csv};

\node[anchor=west, font=\scriptsize, text=black] at (axis cs:-2, 2.5) {\textsf{\textbf{Initialization}}};
\node[anchor=west, font=\scriptsize, text=red!60!black] at (axis cs:60, 2.5) {\textsf{\textbf{Fault 1}}};
\node[anchor=west, font=\scriptsize, text=green!40!black] at (axis cs:110, 2.5) {\textsf{\textbf{Normal}}};
\node[anchor=west, font=\scriptsize, text=blue!50!black] at (axis cs:160, 2.5) {\textsf{\textbf{Fault 2}}};
\node[anchor=west, font=\scriptsize, text=green!40!black] at (axis cs:210, 2.5) {\textsf{\textbf{Normal}}};

\nextgroupplot[black,
    ylabel={$d(\mathcal{B}|_{\mathbf{T}},\Image (H_t))$},
    xlabel={$t$},
    ymin=0, ymax=3,
    ytick distance=0.5,
    clip=false,
    legend style={
        at={(0.03,0.97)},
        anchor=north west,
        font=\scriptsize,
        draw=none,
        fill=none
    }
]

\addplot [draw=none, fill=grayarea, opacity=0.4] coordinates {(0, 0) (24, 0) (24, 3) (0, 3)};
\addplot [draw=none, fill=greenarea, opacity=0.4] coordinates {(24, 0)  (24, 3) (50, 3) (50, 0)};
\addplot [draw=none, fill=redarea, opacity=0.4] coordinates {(50, 0) (100, 0) (100, 3) (50, 3)};
\addplot [draw=none, fill=greenarea, opacity=0.4] coordinates {(100, 0) (150, 0) (150, 3) (100, 3)};
\addplot [draw=none, fill=bluearea, opacity=0.4] coordinates {(150, 0) (200, 0) (200, 3) (150, 3)};
\addplot [draw=none, fill=greenarea, opacity=0.4] coordinates {(200, 0) (250, 0) (250, 3) (200, 3)};
\addplot [const plot, color=blue, very thick] table [x index=0, y index=1, col sep=comma] {distance_chordal.csv};
\addplot [const plot, color=orange, dashed, very thick] table [x index=0, y index=1, col sep=comma] {distance_gap.csv};

\end{groupplot}
\end{tikzpicture}

\caption{Top: Time history of the output signal $y$. Bottom:  Evolution of the distance between the behavior defined by the image of a Hankel matrix $H_t \in \mathbb{R}^{T \times \tau}$ and the nominal behavior $\mathcal{B}|_{\mathbf{T}}$. Distances are computed using the chordal (solid) and $L$-gap (dashed) metrics. Shaded regions indicate initialization (gray), nominal (green), and anomalous (red and blue) intervals. No anomaly detection is performed during initialization. }
\label{fig:anomaly_detection}
\end{figure}

\section{Conclusion} \label{sec:conclusion}

We introduced a class of distances for finite-horizon \gls{LTI} behaviors that remain invariant under coordinate transformations, rotations, and permutations, enabling comparisons across dimensions. The proposed metrics also capture complexity–misfit trade-offs, offering a principled solution to a  modeling puzzle and a geometric interpretation of the \gls{MPUM}. Finally, we showed their practical value in anomaly detection, where they offer  improved resolution over existing distances. Future work should explore their application to time series clustering, system identification with prior knowledge, and data-driven control.

\bibliographystyle{IEEEtran}%
\bibliography{IEEEabrv,refs_min}

\end{document}